\documentclass{article}
\usepackage[latin9]{inputenc}
\usepackage{amsmath}

\makeatletter

\usepackage{amsfonts}

\newtheorem{theorem}{Theorem}\newtheorem{acknowledgement}[theorem]{Acknowledgement}\newtheorem{corollary}[theorem]{Corollary}\newtheorem{definition}[theorem]{Definition}\newtheorem{example}[theorem]{Example}\newtheorem{lemma}[theorem]{Lemma}\newtheorem{proposition}[theorem]{Proposition}\newenvironment{proof}[1][Proof]{\noindent\textbf{#1.} }{\ \rule{0.5em}{0.5em}}
\DeclareMathOperator{\cov}{Cov}
\DeclareMathOperator{\Variance}{Var}

\makeatother

\begin{document}
\title{Bi-log-concavity: some properties and some remarks towards a multi-dimensional
extension}
\author{{\small{}A. Saumard }\\
{\small{} 
Crest-Ensai, Universit{\'e} Bretagne-Loire}}
\maketitle
\begin{abstract}
Bi-log-concavity of probability measures is a univariate extension
of the notion of log-concavity that has been recently proposed in
a statistical literature. Among other things, it has the nice property
from a modelisation perspective to admit some multimodal distributions,
while preserving some nice features of log-concave measures. We compute
the isoperimetric constant for a bi-log-concave measure, extending
a property available for log-concave measures. This implies that bi-log-concave
measures have exponentially decreasing tails. Then we show that the
convolution of a bi-log-concave measure with a log-concave one is
bi-log-concave. Consequently, infinitely differentiable, positive
densities are dense in the set of bi-log-concave densities for $L_{p}$-norms,
$p\in\left[1,+\infty\right]$. We also derive a necessary and sufficient
condition for the convolution of two bi-log-concave measures to be
bi-log-concave. We conclude this note by discussing ways of defining
a multi-dimensional extension of the notion of bi-log-concavity. We
propose an approach based on a variant of the isoperimetric problem,
restricted to half-spaces.
\end{abstract}

\section{Introduction}

\sloppy Bi-log-concavity (of a probability measure on the real line)
is a property recently introduced by D{\"u}mbgen, Kolesnyk and Wilke (\cite{MR3600702}),
that aims at bypassing some restrictive aspects of log-concavity while
preserving some of its nice features. More precisely, bi-log-concavity
amounts to log-concavity of both $F$ and $1-F$ and a simple application
of Pr{\'e}kopa's theorem on stability of log-concavity through marginalization
(\cite{MR0404557}, see also \cite{SauWel2014} for a discussion on
the various proofs of this fundamental theorem) shows that log-concave
measures are also bi-log-concave (see \cite{Bagnoli:95} for a more
direct, elementary proof of this latter fact).

From a modelisation perspective, bi-log-concavity and log-concavity
may be seen as shape constraints. In statistics, when they are available,
shape constraints represent an interesting alternative to more classical
parametric, semi-parametric or non-parametric approaches and constitute
an active contemporary line of research (\cite{MR2757433,MR3881205}).
Bi-log-concavity was indeed proposed in the aim to contribute to this
research area (\cite{MR3600702}). It was used in \cite{MR3600702}
to construct efficient confidence bands for the cumulative distribution
function and some functionals of it. The authors highlight that bi-log-concave
measures admit multi-modal measures while it is well-known that log-concave
measures are unimodal. Furthermore, D{\"u}mbgen et al. \cite{MR3600702}
establish the following characterization of bi-log-concave distributions.
For a distribution function $F$, denote 
\[
J\left(F\right)\equiv\left\{ x\in\mathbb{R}:0<F\left(x\right)<1\right\} 
\]
and call ``non-degenerate'', the functions $F$ such that $J\left(F\right)\neq\emptyset$.

\begin{theorem}[Characterization of bi-log-concavity, \cite{MR3600702}]\label{theorem_charac}Let
$F$ be a non-degenerate distribution function. The following four
statements are equivalent:
\begin{description}
\item [{(i)}] $F$ is bi-log-concave, i.e. $F$ and $1-F$ are log-concave
functions in the sense that their logarithm is concave.
\item [{(ii)}] $F$ is continuous on $\mathbb{R}$ and differentiable on
$J\left(F\right)$ with derivative $f=F^{\prime}$ such that, for
all $x\in J(F)$ and $t\in\mathbb{R}$,
\[
1-\left(1-F(x)\right)\exp\left(-\frac{f(x)}{1-F(x)}t\right)\leq F(x+t)\leq F(x)\exp\left(\frac{f(x)}{F(x)}t\right)\,.
\]
\item [{(iii)}] $F$ is continuous on $\mathbb{R}$ and differentiable
on $J\left(F\right)$ with derivative $f=F^{\prime}$ such that the
hazard function $f/(1-F)$ is non-decreasing and reverse hazard function
$f/F$ is non-increasing on $J(F)$.
\item [{(iv)}] $F$ is continuous on $\mathbb{R}$ and differentiable on
$J\left(F\right)$ with bounded and strictly positive derivative $f=F^{\prime}$.
Furthermore, $f$ is locally Lipschitz continuous on $J\left(F\right)$
with $L_{1}-$derivative $f^{\prime}=F^{\prime\prime}$ satisfying
\[
\frac{-f^{2}}{1-F}\leq f^{\prime}\leq\frac{f^{2}}{F}\,.
\]
\end{description}
\end{theorem}

Note that if one includes degenerate measures - that is Dirac masses
- it is easily seen that the set of bi-log-concave measures is closed
under weak limits.

Just as $s$-concave measures generalize log-concave ones, Laha and
Wellner \cite{wellner2017bi} proposed the concept of bi-$s^{*}$-concavity,
that generalize bi-log-concavity and that include $s$-concave densities.
Some characterizations of bi-$s^{*}$-concavity, that extend the previous
theorem, are derived in \cite{wellner2017bi}.

On the probabilistic side, even if some characterizations are available,
many important questions remain about the properties of bi-log-concave
measures. Indeed, log-concave measures satisfy many nice properties
(see for instance \cite{Guedon:12,SauWel2014,MR3837280} and references
therein) and it is natural to ask whether some of those are extended
to bi-log-concave measures. Answering this question is the primary
object of this note.

We show in Section \ref{sec:Isoperimetry-and-concentration} that
the isoperimetric constant of a bi-log-concave measure is simply equal
to two times the value of its density with respect to the Lebesgue
measure - that indeed exists - at its median, thus extending a property
available for log-concave measures. We deduce that a bi-log-concave
measure has exponential tails, also extending a property valid in
the log-concave case.

In Section \ref{sec:Stability-through-convolution}, we show that
the convolution of a log-concave measure and a bi-log-concave measure
is bi-log-concave. As a consequence, we get that any bi-log-concave
measure can be approximated by a sequence of bi-log-concave measures
having regular densities. Furthermore, we give a necessary and sufficient
condition for the convolution of two bi-log-concave measures to be
bi-log-concave.

Finally, we discuss in Section \ref{sec:Towards-a-generalization}
possible ways to obtain a multivariate notion of bi-log-concavity.
This problem is not a priori obvious, because the definition of bi-log-concavity
in one dimension relies on the cumulative distribution function and
so, on the total order existing on real numbers. To this end, we derive
a characterization of (symmetric) bi-log-concave measures on $\mathbb{R}$
through their isoperimetric profile. Then we propose a multidimensional
generalization for symmetric measures by considering their isoperimetric
profile, restricted to half spaces. We conclude by discussing a way
to strengthen the latter definition in order to ensure stability through
convolution by any log-concave measure. The question of providing
a nice definition of bi-log-concavity in higher dimension, that would
also impose existence of some exponential moments, remains open.

\section{Isoperimetry and concentration for bi-log-concave measures\label{sec:Isoperimetry-and-concentration}}

\sloppypar Let $F\left(x\right)=\mu\left(\left(-\infty,x\right]\right)$
be the distribution function of a probability measure $\mu$ on the
real line. Assume that $\mu$ is non-degenerate (in the sense of its
distribution function being non-degenerate) and let $f$ be the density
of its absolutely continuous part.

Recall the following formula for the isoperimetric constant $Is\left(\mu\right)$
of $\mu$, due to Bobkov and Houdr{\'e} \cite{BobHoud97},
\[
Is\left(\mu\right)=\text{ess}\inf_{x\in J\left(F\right)}\frac{f\left(x\right)}{\min\left\{ F\left(x\right),1-F\left(x\right)\right\} }\text{ .}
\]

The following theorem extends a well-known fact related to isoperimetric
constant for log-concave measure to the case of bi-log-concave measures.

\begin{theorem} \label{theorem_isop_constant}Let $\mu$ be a probability
measure with non-degenerate distribution function $F$ being bi-log-concave.
Then $\mu$ admits a density $f=F^{\prime}$ on $J\left(F\right)$
and it holds
\[
Is\left(\mu\right)=2f\left(m\right)\text{ ,}
\]
where $m$ is the median of $\mu$.\end{theorem}

In general, the isoperimetric constant is hard to compute, but in
the bi-log-concave case Theorem \ref{theorem_isop_constant}\ provides
a straightforward formula, that extends a formula valid for log-concave
measures (see for instance \cite{SauWel2014}).

In the following, we will also use the notation $J\left(F\right)=\left(a,b\right)$.

\begin{proof} Note that the median $m$ is indeed unique by Theorem
\ref{theorem_charac} above. For $x\in\left(a,m\right]$, 
\[
I_{F}\left(x\right):=\frac{f\left(x\right)}{\min\left\{ F\left(x\right),1-F\left(x\right)\right\} }=\frac{f\left(x\right)}{F\left(x\right)}\text{ .}
\]
As $\mu$ is bi-log-concave, $I_{F}$ is thus non-increasing on $\left(a,m\right]$.
For $x\in\left[m,b\right)$,
\[
I_{F}\left(x\right)=\frac{f\left(x\right)}{1-F\left(x\right)}\text{ .}
\]
Thus, $I_{F}$ is non-decreasing on $\left[m,b\right)$. Consequently,
the maximum of $I_{F}\left(x\right)$ is attained on $m$ and its
value is $Is\left(\mu\right)=2f\left(m\right)$. \end{proof}

\begin{corollary} \label{cor_isop_concen}Let $\mu$ as above be
a bi-log-concave measure with median $m$. Then $f\left(m\right)>0$
and $\mu$ satisfies the following Poincar{\'e} inequality: for any square
integrable function $f\in L_{2}\left(\mu\right)$ with derivative
$f^{\prime}\in L_{2}\left(\mu\right)$,
\begin{equation}
f^{2}\left(m\right)\Variance_{\mu}\left(f\right)\leq\int\left(f^{\prime}\right)^{2}d\mu\,,\label{eq:Poin}
\end{equation}
where $\Variance_{\mu}\left(f\right)=\int f^{2}d\mu-\left(\int fd\mu\right)^{2}$
is the variance of $f$ with respect to $\mu$. Consequently, $\mu$
has bounded $\Psi_{1}$ Orlicz norm and achieves the following exponential
concentration inequality,
\begin{equation}
\alpha_{\mu}\left(r\right)\leq\exp\left(-rf\left(m\right)/3\right)\,,\label{eq:concen_gene}
\end{equation}
where $\alpha_{\mu}$ is the concentration function of $\mu$, defined
by $\alpha_{\mu}\left(r\right)=\sup\left\{ 1-\mu\left(A_{r}\right):A\subset\mathbb{R},\mu\left(A\right)\geq1/2\right\} $,
where $r>0$ and $A_{r}=\left\{ x\in\mathbb{R}:\exists y\in A,\left|x-y\right|<r\right\} $
is the (open) $r-$neighborhood of $A$.\end{corollary}

As it is well-known (see \cite{MR1849347} for instance), inequality
(\ref{eq:concen_gene}) implies that for any $1-$Lipschitz function
$f$,
\[
\mu\left(f\geq m_{f}+r\right)\leq\exp\left(-rf\left(m\right)/3\right)\,,
\]
where $m_{f}$ is a median of $f$, that is $\mu\left(f\geq m_{f}\right)\geq1/2$
and $\mu\left(f\geq m_{f}\right)\geq1/2$.

\begin{proof}The fact that $f\left(m\right)>0$ is given by point
\textbf{(iii)} of Theorem \ref{theorem_charac} above. Then Inequality
(\ref{eq:Poin}) is a consequence of Theorem \ref{theorem_isop_constant}
via Cheeger's inequality for the first eigenvalue of the Laplacian
(see for instance Inequality 3.1 in \cite{MR1849347}). Inequality
(\ref{eq:concen_gene}) is a classical consequence of Inequality (\ref{eq:Poin})
as well (see Theorem 3.1 in \cite{MR1849347}).\end{proof}

Note that, following Bobkov \cite{Bobkov:96a}, for a log-concave
probability measure $\mu$ on $\mathbb{R}$ having a positive density
$f$ on $J\left(F\right)$, the function $I\left(p\right)=f\left(F^{-1}\left(p\right)\right)$
is concave. By Theorem \ref{theorem_charac} above, bi-log-concavity
of $\mu$ reduces to non-increasingness of the functions $f/F$ and
$-f/\left(1-F\right)$, which is equivalent to non-increasingness
of $I\left(p\right)/p$ and $-I\left(p\right)/\left(1-p\right)$.
As $I\left(p\right)$ is concave for a log-measure $\mu$ and as $I\left(0\right)=I\left(1\right)=0$,
bi-log-concavity follows. This gives another proof of the fact that
log-concave measures are bi-log-concave.

\begin{example} The function $I\left(p\right)=f\left(F^{-1}\left(p\right)\right)$
is in general hard to compute. But a few easy examples exist. For
instance, for the logistic distribution, $F\left(x\right)=1/\left(1+\exp\left(-x\right)\right)$,
we have $I\left(p\right)=p\left(1-p\right)$. For the Laplace distribution,
$f\left(x\right)=\exp\left(-\left\vert x\right\vert \right)/2$, $I\left(p\right)=\min\left\{ p,1-p\right\} $.
\end{example}

\section{Stability through convolution\label{sec:Stability-through-convolution}}

Take $X$ and $Y$ two independent random variables with respective
distribution functions $F_{X}$ and $F_{Y}$ that are bi-log-concave.
Hence $X$ and $Y$ have densities, denoted by $f_{X}$ and $f_{Y}$.
Then
\begin{equation}
F_{X+Y}\left(x\right)=\mathbb{P}\left(X+Y\leq x\right)\mathbb{=E}\left[\mathbb{P}\left(X\leq x-Y\left\vert Y\right.\right)\right]=\int F_{X}\left(x-y\right)f_{Y}\left(y\right)dy\text{ .}\label{eq:formula_distrib}
\end{equation}
In addition,
\begin{equation}
1-F_{X+Y}\left(x\right)=\int\left(1-F_{X}\left(x-y\right)\right)f_{Y}\left(y\right)dy\text{ .}\label{eq:formula_upper_tail-1}
\end{equation}
\begin{proposition} \label{prop_stab_log_con}If $X$ is bi-log-concave,
$Y$ is log-concave and $X$ is independent from $Y$, then $X+Y$
is bi-log-concave. \end{proposition}

\begin{proof} By using formulas (\ref{eq:formula_distrib}) and (\ref{eq:formula_upper_tail-1}),
this is a direct application of Pr{\'e}kopa's theorem (\cite{MR0404557})
on the marginal of a log-concave function. \end{proof}

\begin{corollary} \label{corollary_approx}Take a (non-degenerate)
bi-log-concave measure on $\mathbb{R}$, with density $f$. Then there
exists a sequence of infinitely differentiable bi-log-concave densities,
positive on $\mathbb{R}$, that converge to $f$ in $L_{p}\left(\textrm{Leb}\right)$,
for any $p\in\left[1,+\infty\right].$ \end{corollary}

Corollary \ref{corollary_approx} is also an extension of an approximation
result available in the set of log-concave distributions, see \cite[Section 5.2]{SauWel2014}.

\begin{proof}It suffices to consider the convolution of $f$ with
a sequence of centered Gaussian densities with variances converging
to zero. As $f$ has an exponential moment, it belongs to any $L_{p}\left(\textrm{Leb}\right)$,
$p\in\left[1,+\infty\right].$ Then a simple application of classical
theorems about convolution in $L_{p}$ (see for instance \cite[p. 148]{MR924157})
allows to check that the approximations converge to $f$ in any $L_{p}\left(\textrm{Leb}\right)$,
$p\in\left[1,+\infty\right].$ \end{proof}

More generally, the following theorem gives a necessary and sufficient
condition for the convolution of two bi-log-concave measures to be
bi-log-concave.

\begin{theorem}\label{theorem_stabi_convol}Take $X$and $Y$ two
independent bi-log-concave random variables with respective densities
$f_{X}$ and $f_{Y}$ and cumulative distribution functions $F_{X}$
and $F_{Y}$. Denote $w\left(x,y\right)=f_{Y}\left(y\right)F_{X}\left(x-y\right)$
and $\bar{w}\left(x,y\right)=f_{Y}\left(y\right)\left(1-F_{X}\right)\left(x-y\right)$
and consider for any $x$$\in J\left(F_{X+Y}\right)$, the following
measures on $\mathbb{R}$,
\[
dm_{x}\left(y\right)=\frac{w\left(x,y\right)dy}{\int w\left(x,y\right)dy}=\frac{w\left(x,y\right)dy}{F_{X+Y}\left(x\right)}
\]
and
\[
d\bar{m}_{x}\left(y\right)=\frac{\bar{w}\left(x,y\right)dy}{\int\bar{w}\left(x,y\right)dy}=\frac{\bar{w}\left(x,y\right)dy}{1-F_{X+Y}\left(x\right)}\,.
\]
Then $X+Y$ is bi-log-concave if and only if for any $x$$\in J\left(F_{X+Y}\right)$,
\begin{equation}
\cov_{m_{x}}\left(\left(-\log f_{Y}\right)^{\prime},\left(-\log F_{X}\right)^{\prime}\left(x-\cdot\right)\right)\geq0\label{eq:cond_cov_1}
\end{equation}
and
\begin{equation}
\cov_{\bar{m}_{x}}\left(\left(-\log f_{Y}\right)^{\prime},\left(-\log\left(1-F_{X}\right)\right)^{\prime}\left(x-\cdot\right)\right)\geq0\,.\label{eq:cond_cov_2}
\end{equation}
\end{theorem}

Of course, a simple symmetrization argument shows that conditions
(\ref{eq:cond_cov_1}) and (\ref{eq:cond_cov_2}) are satisfied if
$\left(-\log f_{Y}\right)^{\prime\prime}\geq0$ pointwise, which means
that $f_{Y}$ is log-concave, in which case we recover Proposition
\ref{prop_stab_log_con}\ above. But Theorem \ref{theorem_stabi_convol}
is more general. Indeed, it is easily checked by direct computations
that the convolution the Gaussian mixture $2^{-1}\mathcal{N}\left(-1.34,1\right)+2^{-1}\mathcal{N}\left(1.34,1\right)$
- which is bi-log-concave but not log-concave, see \cite[Section 2]{MR3600702}
- with itself is bi-log-concave.

To prove Theorem \ref{theorem_stabi_convol}, we will use the following
lemma.

\begin{lemma}\label{lemma_expect_cov}Take $p,q\in\left[1,+\infty\right]$
such that $p^{-1}+q^{-1}=1$ and a measure $\nu$ on $\mathbb{R}$
with density $f=\exp\left(-\phi\right)$ absolutely continuous and
$f^{\prime}\in L_{p}\left(\nu\right)$. Take $g\in L_{q}\left(\nu\right)$
Lipschitz continuous such that $g^{\prime}\in L_{1}\left(\nu\right)$
and 
\[
\lim_{x\rightarrow+\infty}f\left(x\right)\left(g\left(x\right)-\mathbb{E}_{\nu}\left[g\right]\right)=\lim_{x\rightarrow-\infty}f\left(x\right)\left(g\left(x\right)-\mathbb{E}_{\nu}\left[g\right]\right)=0,
\]
then
\[
\mathbb{E}_{\nu}\left[g^{\prime}\right]=\cov_{\nu}\left(g,\phi^{\prime}\right).
\]
 \end{lemma}

\begin{proof}[Proof of Lemma \ref{lemma_expect_cov}]This a simple
integration by parts: from the assumptions, we have
\[
\mathbb{E}_{\nu}\left[g^{\prime}\right]=\int g^{\prime}fdx=\int\left(g-\mathbb{E}_{\nu}\left[g\right]\right)f^{\prime}dx=\int\left(g-\mathbb{E}_{\nu}\left[g\right]\right)\phi^{\prime}fdx.
\]

\end{proof}

\begin{proof}[Proof of Theorem \ref{theorem_stabi_convol}]Recall
that we have
\[
F_{X+Y}\left(x\right)=\int f_{Y}\left(y\right)F_{X}\left(x-y\right)dy=\int w\left(x,y\right)dy\text{ .}
\]
Our first goal is to find some conditions such that $F_{X+Y}$ is
log-concave. It is sufficient to prove that, for any $x\in J\left(F_{X+Y}\right)$,
\[
\frac{\left(F_{X+Y}^{\prime}\left(x\right)\right)^{2}}{F_{X+Y}\left(x\right)}-F_{X+Y}^{\prime\prime}\left(x\right)\geq0\,,
\]
or equivalently,
\[
\left(\frac{F_{X+Y}^{\prime}\left(x\right)}{F_{X+Y}\left(x\right)}\right)^{2}-\frac{F_{X+Y}^{\prime\prime}\left(x\right)}{F_{X+Y}\left(x\right)}\geq0\,.
\]
Denote $\rho_{X}=\left(\log F_{X}\right)^{\prime}$. We have
\begin{eqnarray*}
F_{X+Y}\left(x\right) & = & \int w\left(x,y\right)dy\\
f_{X+Y}\left(x\right) & = & F_{X+Y}^{\prime}\left(x\right)=\int\rho_{X}\left(x-y\right)w\left(x,y\right)dy\\
F_{X+Y}^{\prime\prime}\left(x\right) & = & \int\left(\rho_{X}^{\prime}\left(x-y\right)+\rho_{X}^{2}\left(x-y\right)\right)w\left(x,y\right)dy
\end{eqnarray*}
Furthermore, we get
\[
\left(\frac{F_{X+Y}^{\prime}\left(x\right)}{F_{X+Y}\left(x\right)}\right)^{2}-\frac{\int w\rho_{X}^{2}\left(x-y\right)dy}{F_{X+Y}\left(x\right)}=-\Variance_{m_{x}}\left(\rho_{X}\left(x-\cdot\right)\right)\text{ .}
\]
Now, by Lemma \ref{lemma_expect_cov}, it holds,
\begin{align*}
\frac{\int\rho_{X}^{\prime}\left(x-y\right)w\left(x,y\right)dy}{F_{X+Y}\left(x\right)} & =\mathbb{E}_{m_{x}}\left[\rho_{X}^{\prime}\left(x-\cdot\right)\right]\\
 & =\cov_{m_{x}}\left(-\rho_{X}\left(x-\cdot\right),\left(-\log f_{Y}\right)^{\prime}+\rho_{X}\left(x-\cdot\right)\right)\,.
\end{align*}
Gathering the equations, we get
\begin{align*}
\left(\frac{F_{X+Y}^{\prime}\left(x\right)}{F_{X+Y}\left(x\right)}\right)^{2}-\frac{F_{X+Y}^{\prime\prime}\left(x\right)}{F_{X+Y}\left(x\right)} & =\cov_{m_{x}}\left(-\rho_{X}\left(x-\cdot\right),\left(-\log f_{Y}\right)^{\prime}\right)\\
 & =\cov_{m_{x}}\left(-\log F_{X}\left(x-\cdot\right),\left(-\log f_{Y}\right)^{\prime}\right)\,,
\end{align*}
which gives condition (\ref{eq:cond_cov_1}). Likewise condition (\ref{eq:cond_cov_2})
arises from the same type of computations when studying log-concavity
of $\left(1-F_{X+Y}\right)$. \end{proof}

\subsection{Towards a multivariate notion of bi-log-concavity\label{sec:Towards-a-generalization}}

Let us introduce this section with the following remark. The isoperimetric
profile $I_{\mu}$ is defined as follows: for any $p\in\left(0,1\right)$,
\[
I_{\mu}\left(p\right)=\inf_{\mu\left(A\right)=p}\mu^{+}\left(A\right)\text{ ,}
\]
where $\mu^{+}\left(A\right)=\lim\inf_{h\rightarrow0^{+}}\left(\mu\left(A^{h}\right)-\mu\left(A\right)\right)/h$,
with $A^{h}=\left\{ x\in\mathbb{R}:\delta\left(x,A\right)<h\right\} $.
Note that the isoperimetric profile $I_{\mu}$ depends on the distance
$\delta$ that is considered. Unless explicitly mentioned, we will
consider in the following that the distance $\delta$ is the Euclidean
distance. From inequality (2.1) in \cite{Bobkov:96a}, we have for
a log-concave measure $\mu$ on $\mathbb{R}$ and any $h>0$,
\begin{eqnarray*}
\min_{\mu\left(A\right)\geq p}\mu\left(A+\left(-h,h\right)\right) & = & \min_{\mu\left(A\right)=p}\mu\left(A+\left(-h,h\right)\right)\\
 &  & =\min\left\{ F\left(F^{-1}\left(p\right)+h\right),1-F\left(F^{-1}\left(1-p\right)-h\right)\right\} \text{ .}
\end{eqnarray*}
Hence,
\begin{eqnarray*}
I_{\mu}\left(p\right) & = & \min\left\{ f\left(F^{-1}\left(p\right)\right),f\left(F^{-1}\left(1-p\right)\right)\right\} \\
 & = & \underset{A\text{ half-space of }\mathbb{R}}{\inf_{\mu\left(A\right)=p}}\mu^{+}\left(A\right)\text{ .}
\end{eqnarray*}
If $\mu$ is moreover symmetric, then $I_{\mu}\left(p\right)=f\left(F^{-1}\left(p\right)\right)$
for any $p\in\left(0,1\right)$. 

For a general measure $\mu$, we define the isoperimetric profile
restricted to half-spaces $I_{\mu}^{H}$: for any $p\in\left(0,1\right)$,
\[
I_{\mu}^{H}\left(p\right)=\underset{A\text{ half-space}}{\inf_{\mu\left(A\right)=p}}\mu^{+}\left(A\right)\text{ .}
\]
For a measure $\mu$ on $\mathbb{R}$ having a density $f$, one has
\[
I_{\mu}^{H}\left(p\right)=\min\left\{ f\left(F^{-1}\left(p\right)\right),f\left(F^{-1}\left(1-p\right)\right)\right\} \,,
\]
since possible half-spaces are in this case $\left(-\infty,x_{A}\right]$
or $\left[x_{A},+\infty\right)$. If $\mu$ is symmetric, then $I_{\mu}^{H}\left(p\right)=f\left(F^{-1}\left(p\right)\right)$
and bi-log-concavity is equivalent to non-increasingness of $p\mapsto I_{\mu}^{H}\left(p\right)/p$
on $\left(0,1\right)$. As proved in Bobkov \cite[Proposition A.1]{Bobkov:96a},
log-concavity on $\mathbb{R}$ is actually equivalent to concavity
of $f\left(F^{-1}\left(p\right)\right)$. 

Furthermore, as previously remarked, $I_{\mu}\equiv I_{\mu}^{H}$
in the one-dimensional log-concave case. The latter identity is still
true in higher dimension for the Gaussian measure when the distance
is given by the Euclidean norm (see \cite{MR0399402}) and this characterizes
Gaussian measures. In general, it also holds $I_{\mu}\leq I_{\mu}^{H}$
pointwise.

Take the distance $\delta$ to be given by the sup-norm, $\delta\left(x,y\right)=\left\Vert x-y\right\Vert _{\infty}$.
Then, for any set $A\subset\mathbb{R}$, we have $A^{h}=A+h\left[-1,1\right]^{n}$.
In this case, Bobkov \cite[Theorem 1.1]{Bobkov:96a} characterizes
symmetric log-concave measures for which $I_{\mu}^{H}=I_{\mu}$. 

Reverse relation in higher dimension between $I_{\mu}$ and $I_{\mu}^{H}$
when $\mu$ is log-concave, is related to the so-called KLS-hyperplane
conjecture (see for instance \cite{DBLP:conf/focs/LeeV17,DBLP:journals/corr/LeeV18}). 

A possible extension of the notion of bi-log-concavity is the following.

\begin{definition}\label{def_weak_blc} Let $\mu$ be a probability
measure on $\mathbb{R}^{d},$ $d\geq1$. Assume that $\mu$ is symmetric
around the origin. Then $\mu$ is said to be weakly bi-log-concave
(with respect to the distance $\delta$) if the function 
\[
p\mapsto I_{\mu}^{H}\left(p\right)/p
\]
is non-increasing on $\left(0,1\right)$. \end{definition}

The latter definition extends the definition of bi-log-concavity for
symmetric measures on the real line. However, we consider that the
definition is ``weak'' since, as we will see, it seems in fact natural
to ask for more. In the following, a symmetric measure is a measure
that is symmetric around the origin.

\begin{proposition}Symmetric log-concave measures on $\mathbb{R}^{d}$
are bi-log-concave (for the Euclidean distance).\end{proposition}

\begin{proof}Take $u$ a unit vector and consider the measure $\mu_{u}$
defined to be the projection of the measure $\mu$ on the line containing
$0$ and directed by $u$. Consequently, $\mu_{u}$ is a log-concave
measure on $\mathbb{R}$, symmetric around zero. Hence $I_{\mu_{u}}\left(p\right)$
is concave and consequently, $I_{\mu_{u}}\left(p\right)/p$ is nonincreasing.
Since half-spaces are parameterized by unitary vectors together with
a point on the line containing zero and directed by the considered
unitary vector, this readily gives the nonincreasingness of $I_{\mu}\left(p\right)/p$.\end{proof}

One can notice that the latter proof is in fact only based on stability
of log-concavity through one-dimensional marginalizations. This naturally
leads to the following second definition of bi-log-concavity in higher
dimension. 

\begin{definition}Let $\mu$ be a probability measure on $\mathbb{R}^{d},$
$d\geq1$. Then $\mu$ is said to be weakly$-*$ bi-log-concave if
for every line $\ell\subset\mathbb{R}^{d}$, the (Euclidean) projection
measure $\mu_{\ell}$ of $\mu$ onto the line $\ell$ is a (one-dimensional)
bi-log-concave measure on $\ell$ (that can be possibly degenerate).
More explicitly, for any $x\in\ell$ and any Borel set $B\subset\mathbb{R}$,
\[
\mu_{\ell}\left(x+Bu\right)=\mu\left\{ y\in\mathbb{R}^{d}:\left(y-x\right)\cdot u\in B\right\} 
\]
where $u$ is a unit directional vector of the line $\ell$.\end{definition}

Note that weakly$-*$ bi-log-concave measures are not necessarily
symmetric. In the case of symmetric measures, the notion of weakly$-*$
bi-log-concavity is actually a strengthening of Definition \ref{def_weak_blc}.

\begin{proposition}\label{prop_weak_weak*}Let $\mu$ be a symmetric,
weakly$-*$ bi-log-concave probability measure on $\mathbb{R}^{d},$
$d\geq1$. Then $\mu$ is weakly bi-log-concave.\end{proposition}

\begin{proof}By parametrization of half-spaces, we have the following
formula, for any $p\in$$\left(0,1\right)$,
\[
I_{\mu}^{H}\left(p\right)=\inf_{\textrm{lines}\,\ell;0\in\ell}I_{\mu_{\ell}}^{H}\left(p\right)\,.
\]
Then the conclusion follows by noticing that for any line $\ell$
such that $0\in\ell$, the projection measure $\mu_{\ell}$ of $\mu$
is symmetric and bi-log-concave. Hence, $I_{\mu_{\ell}}^{H}\left(p\right)/p$
is non-increasing on $\left(0,1\right)$ and so is $I_{\mu}^{H}\left(p\right)/p$.\end{proof}

The following result states that the notion of weakly$-*$ bi-log-concavity
is stable through convolution by log-concave measures.

\begin{proposition}The convolution of a log-concave measure with
a weakly$-*$ bi-log-concave one is weakly$-*$ bi-log-concave. \end{proposition}

\begin{proof}The formula $\left(X+Y\right)\cdot u=X\cdot u+Y\cdot u$
shows that the projection of the convolution of two measures on a
line is the convolution of the projections of measures on this line.
This allows to reduce the stability through convolution by a log-concave
measure to dimension one and concludes the proof.\end{proof}

As for the log-concave case, it is moreover directly seen that weakly
and weakly$-*$ log-concavity are stable by affine transformations
of the space.

Actually, in addition to containing log-concave measures and being
stable through convolution by a log-concave measure, there are at
least two other properties that one would naturally require for a
convenient multidimensional concept of bi-log-concavity: existence
of a density with respect to the Lebesgue measure on the convex hull
of its support and existence of a finite exponential moment for the
(Euclidean) norm. We can express this latter remark through the following
open problem, that concludes this note.
\begin{flushleft}
\textbf{Open Problem:} Find a nice characterization of probability
measures on $\mathbb{R}^{d}$ that are weakly$-*$ bi-log-concave,
that admit a density with respect to the Lebesgue measure on the convex
hull of their support and whose Euclidean norm has exponentially decreasing
tails.
\par\end{flushleft}

\begin{acknowledgement}I express my deepest gratitude to Jon Wellner,
who introduced me to the intriguing notion of bi-log-concavity and
who provided several numerical computations during his visit at the
Crest-Ensai, that helped to understand the convolution problem for
bi-log-concave measures. Many thanks also to Jon for his comments
on a previous version of this note.\end{acknowledgement}

\bibliographystyle{amsalpha}
\bibliography{chern}
{\footnotesize{} }{\footnotesize\par}
\end{document}